\documentclass[12pt]{amsart}
\usepackage[english]{babel}
\usepackage[utf8]{inputenc}
\usepackage{amsmath, amssymb, amsthm, amsfonts, amsxtra, stmaryrd, enumerate, graphicx}
\usepackage[colorinlistoftodos]{todonotes}
\usepackage{placeins}

\newtheorem{theorem}{Theorem}[section]

\newtheorem{claim}[theorem]{Claim}

\newtheorem{lemma}[theorem]{Lemma}
\newtheorem{fact}[theorem]{Fact}
\newtheorem{proposition}[theorem]{Proposition}

\newtheorem{question}[theorem]{Question}

\usepackage{tikz-cd}
\newtheorem{definition}{Definition}[section]

\theoremstyle{definition}

\newtheorem{remark}[theorem]{Remark}
\newtheorem{example}[theorem]{Example}

\def\Ind#1#2{#1\setbox0=\hbox{$#1x$}\kern\wd0\hbox to 0pt{\hss$#1\mid$\hss}
\lower.9\ht0\hbox to 0pt{\hss$#1\smile$\hss}\kern\wd0}

\def\ind{\mathop{\mathpalette\Ind{}}}

\def\notind#1#2{#1\setbox0=\hbox{$#1x$}\kern\wd0
\hbox to 0pt{\mathchardef\nn=12854\hss$#1\nn$\kern1.4\wd0\hss}
\hbox to 0pt{\hss$#1\mid$\hss}\lower.9\ht0 \hbox to 0pt{\hss$#1\smile$\hss}\kern\wd0}

\def\nind{\mathop{\mathpalette\notind{}}}

\usepackage{etoolbox}
\patchcmd{\subsection}{-.5em}{.5em}{}{}

\title{On the properties $\mathrm{SOP}_{2^{n+1}+1}$}
\author{Scott Mutchnik}

\begin{document}

\begin{abstract}
We show that approximations of strict order can calibrate the fine structure of genericity. Particularly, we find exponential behavior within the $\mathrm{NSOP}_{n}$ hierarchy from model theory. Let $\ind^{\eth^{0}}$ denote forking-independence. Inductively, a formula \textit{$(n+1)$-$\eth$-divides} over $M$ if it divides by every $\ind^{\eth^{n}}$-Morley sequence over $M$, and \textit{$(n+1)$-$\eth$-forks} over $M$ if it implies a disjunction of formulas that $(n+1)$-$\eth$-divide over $M$; the associated independence relation over models is denoted $\ind^{\eth^{n+1}}$. We show that a theory where $\ind^{\eth^n}$ is symmetric or transitive must be $\mathrm{NSOP}_{2^{n+1}+1}$. We then show that, in the classical examples of $\mathrm{NSOP}_{2^{n+1}+1}$ theories, $\ind^{\eth^{n}}$ is symmetric and transitive; in particular, there are strictly $\mathrm{NSOP}_{2^{n+1}+1}$ theories where $\ind^{\eth^{n}}$ is symmetric and transitive, leaving open the question of whether symmetry or transitivity of $\ind^{\eth^{n}}$ is equivalent to $\mathrm{NSOP}_{2^{n+1}+1}$.

\end{abstract}
\maketitle

\section{Introduction}
This paper is on Shelah's \textit{strong order property} hierarchy, the properties $\mathrm{SOP}_{n}$ introduced in \cite{She95} and extended in \cite{Sh99}, \cite{DS04}. For $n \geq 3$, these are defined as follows:

\begin{definition}
A theory $T$ is $\mathrm{NSOP}_{n}$ (that is, does not have the \emph{n-strong order property}) if there is no definable relation $R(x_{1}, x_{2})$ with no $n$-cycles, but with tuples $\{a_{i}\}_{i \in \omega}$ with $\models R(a_{i}, a_{j})$ for $i <j$. Otherwise it is $\mathrm{SOP}_{n}$.
\end{definition}

For $1 \leq n \leq 4$, these properties have been developed to various degrees. In \cite{KR17}, Kaplan and Ramsey extend the theory of forking-independence in simple theories to \textit{Kim-independence} in \textit{$\mathrm{NSOP}_{1}$ theories}, by modifying the definition of dividing to require an invariant Morley sequence to witness the dividing. There, a characterization of $\mathrm{NSOP}_{1}$ is given in terms of symmetry for Kim-independence, using work of Chernikov and Ramsey in \cite{CR15}, and also in terms of the independence theorem for Kim-independence. Later work has continued the development of Kim-independence in $\mathrm{NSOP}_{1}$ theories; for example, see Kaplan and Ramsey (\cite{KR19}) for transitivity and witnessing; Kaplan, Ramsey and Shelah (\cite{KRS19}) for local character; Dobrowolski, Kim and Ramsey (\cite{DKR22}) and Chernikov, Kim and Ramsey (\cite{CKR20}) for independence over sets; Kruckman and Ramsey (\cite{KR18}) and Kruckman, Tran and Walsberg (\cite{KTW22}) for improvements upon the independence theorem; Kim (\cite{K21}) for canonical bases; and Kamsma (\cite{K22}), Dobrowolski and Kamsma (\cite{DKR22}) and Dmitrieva, Gallinaro and Kamsma (\cite{DGK23}) for extensions to positive logic, as well as the examples, by various authors, of $\mathrm{NSOP}_{1}$ theories in applied settings. See also Kim and Kim \cite{KK11}, Chernikov and Ramsey (\cite{CR15}), Ramsey (\cite{R19}), and Casanova and Kim (\cite{CK19}) for type-counting and combinatorial criteria for $\mathrm{SOP}_{1}$ and $\mathrm{SOP}_{2}$, and Ahn and Kim (\cite{AK20}) for connections of $\mathrm{SOP}_{1}$ and $\mathrm{SOP}_{2}$ to the antichain tree property further developed by Ahn, Kim and Lee in \cite{AK21}. 

The \textit{$\mathrm{SOP}_{2}$ theories} were characterized by Džamonja and Shelah (\cite{DS04}), Shelah and Usvyatsov (\cite{SD00}), and Malliaris and Shelah (\cite{MS17}) as (under mild set-theoretic assumptions) the maximal class in the order $\lhd^{*}$, related to Keisler's order, and in celebrated work of Malliaris and Shelah (\cite{MS16}), they were shown to be maximal in Keisler's order, in $\mathrm{ZFC}$. Then in \cite{NSOP2}, it was shown that a theory is $\mathrm{NSOP}_{2}$ if and only if it is $\mathrm{NSOP}_{1}$, bringing together Kim-independence and Keisler's order. It remains open whether all $\mathrm{NSOP}_{3}$ theories are $\mathrm{NSOP}_{2}$

Generalizing work of Evans and Wong in \cite{EW09}, showing the $\omega$-categorical Hrushovski constructions introduced in \cite{Ev02}, which have a natural notion of free amalgamation, are either simple or $\mathrm{SOP}_{3}$, and work of Conant in \cite{Co15} showing that all modular free amalgamation theories are simple or $\mathrm{SOP}_{3}$, the author in \cite{GFA} isolates two structural properties, with no known $\mathrm{NSOP}_{4}$ counterexamples, which generalize \cite{EW09} and \cite{Co15} and imply that a theory must be either $\mathrm{NSOP}_{1}$ or $\mathrm{SOP}_{3}$. As a consequence, all free amalgamation theories are $\mathrm{NSOP}_{1}$ or $\mathrm{SOP}_{3}$. Malliaris and Shelah (\cite{MS17}) show symmetric inconsistency for \textit{higher formulas} in $\mathrm{NSOP}_{3}$ theories, and Malliaris (\cite{Mal10b}) investigates the graph-theoretic depth of independence in relation to $\mathrm{SOP}_{3}$. In \cite{CKR23}, Ramsey, Kaplan and Simon show very recently that all binary $\mathrm{NSOP}_{3}$ theories are simple, by giving a theory of independence for a class of theories containing all binary theories. Until recently, no consequences of $\mathrm{NSOP}_{n}$ were known for the program of further extending the theory of Kim-independence in $\mathrm{NSOP}_{2} = \mathrm{NSOP}_{1}$ theories to $\mathrm{NSOP}_{n}$ for $n > 2$; then the author, in \cite{INDNSOP3}, shows that types in $\mathrm{NSOP}_{3}$ theories with internally $\mathrm{NSOP}_{1}$ structure satisfy Kim's lemma at an external level, as well as an independence theorem, and also shows that $\mathrm{NSOP}_{3}$ theories with symmetric Conant-independence satisfy an independence theorem for finitely satisfiable types with the same Morley sequences, related to that proposed for $\mathrm{NTP}_{2}$ theories by Simon (\cite{Sim20}).

Shelah, in \cite{She95}, gives results on universal models in $\mathrm{SOP}_{4}$ theories. Generalizing a line of argument from the literature originally used by Patel (\cite{Pat06}), Conant, in \cite{Co15} (where an historical overview of this argument can be found), shows free amalgamation theories are $\mathrm{NSOP}_{4}$. In \cite{GFA}, the author connects this result to a potential theory of independence in $\mathrm{NSOP}_{4}$ theories, defining the relation of \textit{Conant-independence}\footnote{This was originally introduced under a nonstandard definition in \cite{NSOP2}, to show $\mathrm{NSOP}_{2}$ theories are $\mathrm{NSOP}_{1}$.}:

\begin{definition}\label{5-conantindependence}
Let $M$ be a model and $\varphi(x, b)$ a formula. We say $\varphi(x, b)$ \emph{Conant-divides} over $M$ if for \emph{every} invariant Morley sequence $\{b_{i}\}_{i \in \omega}$ over $M$ starting with $b$, $\{\varphi(x, b)\}_{i \in \omega}$ is inconsistent. We say $\varphi(x, b)$ \emph{Conant-forks} over $M$ if and only if it implies a disjunction of formulas Conant-dividing over $M$. We say $a$ is \emph{Conant-independent} from $b$ over $M$, written $a \ind^{K^{*}}_{M}b$, if $\mathrm{tp}(a/Mb)$ does not contain any formulas Conant-forking over $M$.
\end{definition}

By Kim's lemma (Theorem 3.16 of \cite{KR17}), this coincides with Kim-independence in $\mathrm{NSOP}_{1}$ theories. Conant-independence gives a plausible theory of independence for $\mathrm{NSOP}_{4}$ theories:

\begin{fact}\label{5-conantnsop4}(Theorem 6.2, \cite{GFA})
Any theory where Conant-independence is symmetric is $\mathrm{NSOP}_{4}$, and there are strictly $\mathrm{NSOP}_{4}$ ($\mathrm{NSOP}_{4}$ and $\mathrm{SOP}_{3}$) theories where Conant-independence is symmetric. Thus $n = 4$ is the largest value of $n$ so that there are strictly $\mathrm{NSOP}_{n}$ theories where Conant-independence is symmetric.
\end{fact}

In \cite{GFA}, the author characterizes Conant-independence in most of the known examples of $\mathrm{NSOP}_{4}$ theories, where it is symmetric. This leaves open the question of whether Conant-independence is symmetric in all $\mathrm{NSOP}_{4}$ theories, giving a full theory of independence for the class $\mathrm{NSOP}_{4}$ theories.

However, to our knowledge, other than some examples (\cite{She95}, \cite{CW04}; see also \cite{CT16}), little has been known about the properties $\mathrm{SOP}_{n}$ for $n \geq 5$.

\:

The main result of this paper will be to generalize the interactions betwen $\mathrm{SOP}_{4}$ and Conant-independence to the higher levels of the $\mathrm{SOP}_{n}$ hierarchy. As with Conant-independence, we will move from the forking-independence ``at a generic scale" considered by Kruckman and Ramsey's work on Kim-independence in (\cite{KR17}, where the phrase is coined), to forking-independence at a \textit{maximally generic scale}, grounding our notion of independence in dividing with respect to \textit{every} Morley sequence of a certain kind, rather than just some Morley sequence. There is precedent for this kind of definition in the ``strong Kim-dividing" of Kaplan, Ramsey and Shelah in \cite{KRS19}, defined in the context of ``dual local character" in $\mathrm{NSOP}_{1}$ theories and grounding the defintion of Conant-independence.

We will also turn our attention to the \textit{fine structure} of the genericity in the sequences that witness dividing, taking into account the variation between different classes of Morley sequences. For Kim-independence in $\mathrm{NSOP}_{1}$ theories, this fine structure is submerged: by Corollary 5.4 of \cite{KR19}, Kim-independence in $\mathrm{NSOP}_{1}$ theories remains the same when one replaces invariant Morley sequences in the genericity with Kim-independence itself. In the examples of $\mathrm{NSOP}_{4}$ theories where Conant-independence has been characterized, it can also be seen that Conant-independence remains the same when one replaces invariant Morley sequences in the definition (Definition \ref{5-conantindependence}) with Conant-nonforking Morley sequences; see remarks at the end of Section 2 of this paper. However, in, say, strictly $\mathrm{NSOP}_{5}$ theories, Conant-independence cannot be symmetric, but a symmetric notion of independence can be obtained in some examples by replacing the invariant Morley sequences with nonforking Morley sequences. More generally, we can iteratively obtain different levels of genericity, the independence relations $\ind^{\eth^{n}}$ defined in Definition \ref{5-eth}. The main result of this paper will be to show, within the interaction between the layers $\ind^{\eth^{n}}$ of genericity and the approximations $\mathrm{SOP}_{k}$ of strict order, the resonance of the exponential function $2^{n+1}+1$.

We show:

\begin{theorem}
     Let $n \geq 1$. If $\ind^{\eth^{n}}$ is symmetric in the theory $T$, then $T$ is $\mathrm{NSOP}_{2^{n+1}+1}$. Moreover, there exists an $\mathrm{SOP}_{2^{n+1}}$ theory in which $\ind^{\eth^{n}}$ is symmetric. So $k = 2^{n+1}+1$ is the largest value of $k$ so that there is an $\mathrm{NSOP}_{k}$ but $\mathrm{SOP}_{k-1}$ theory where $\ind^{\eth^{n}}$ is symmetric.
\end{theorem}

Similar results are proven for (left and right) transitivity. As with Conant-independence, this leaves open the question of whether $\ind^{\eth^{n}}$ is symmetric in all $\mathrm{NSOP}_{2^{n+1}+1}$ theories, which would give a full theory of independence throughout the strict order hierarchy.

In Section 2, we define $\ind^{\eth^{n}}$, show some basic properties necessary for our main result, and give some connections with stability motivating the possibility that symmetry for $\ind^{\eth^{n}}$ forms a hierarchy.

In section 3, we characterize $\ind^{\eth^n}$ in the classical examples of $\mathrm{NSOP}_{2^{n+1}+1}$ theories, including the generic directed graphs without short directed cycles and undirected graphs without short odd cycles of \cite{She94}, and the free roots of the complete graph of \cite{CW04}, developed in \cite{CT16}. Though $\ind^{\eth^n}= \ind^{a}$ will be trivial (and thus symmetric) in these classical examples, giving us the existence result of our main theorem, it is promising for the full characterization, Question \ref{5-openquestion}, that the cycle-free examples and the free roots of the complete graph have $\ind^{\eth^n}= \ind^{a}$ for different reasons. In the cycle-free examples, successive approximations of forking-independence tend towards larger graph-theoretic distances, while in the free roots of the complete graph, distances in successive approximations of forking-independence tend \textit{away} from the extremes.

In section 4, we show that if $\ind^{\eth^{n}}$ is symmetric in the theory $T$, then $T$ is $\mathrm{NSOP}_{2^{n+1}+1}$, completing our main result. We pose the converse as an open question.

\section{Definitions and basic properties}

We recall the defintion of \textit{forking-independence}:

\begin{definition}
    A formula $\varphi(x, b)$ \emph{divides} over a model $M$ if there is an $M$-indiscernible sequence $\{b_{i}\}_{i < \omega}$ with $b_0 = b$ and $\{\varphi(x, b_{i})\}_{i < \omega}$ inconsistent. A formula $\varphi(x, b)$ \emph{forks} over a model $M$ if there are $\varphi_{i}(x, b_{i})$ dividing over $M$ so that $\models \varphi(x, b) \rightarrow \bigvee_{i=1}^{n} \varphi_{i}(x, b_{i})$. We say that $a$ is \emph{forking-independent} from $b$ over $M$, denoted $a \ind_{A}^{f} b$, if $\mathrm{tp}(a/Mb)$ contains no formulas forking over $M$.
\end{definition}

We define the relations $\ind^{\eth^n}$, in analogy with the \textit{Conant-independence} of \cite{GFA}. To give the definition, we need to generalize the notion of Morley sequence to any relation between sets over a model.

\begin{definition}
    Let $\ind$ be a relation between sets over a model. An \textit{$\ind$-Morley sequence} over $M$ is an $M$-indiscernible sequence $\{b_{i}\}_{i <\omega}$ with $b_{i} \ind_{M} b_{0} \ldots b_{i-1}$ for $i < \omega$.
\end{definition}

Let $a\ind^{u}_{M} b$ denote that $\mathrm{tp}(a/Mb)$ is $M$-finitely satisfiable; as elsewhere in the literature, a \textit{finiteley satisfiable} or \textit{coheir Morley sequence} will be a $\ind^{u}$-Morley sequence.

\begin{definition} \label{5-eth}
(1) Let $\ind^{\eth^{0}}$, \emph{$0$-$\eth$-independence}, denote forking-independence over a model $M$; a formula $0$-$\eth$-divides ($0$-$\eth$-forks) over $M$ if it divides (forks) over $M$.

Inductively,

(2a) A formula $\varphi(x, b)$ \textit{$(n+1)$-$\eth$-divides} over a model $M$ if, for any $\ind^{\eth^n}$-Morley sequence $\{b_{i}\}_{i <\omega}$ with $b_{0} = b$, $\{\varphi(x, b_{i})\}_{i < \omega}$ is inconsistent.\footnote{It is not immediate that this defintion is independent of adding or removing unusued parameters in $b$, though this is corrected by the definition of $n+1$-$\eth$-forking. We fix the convention that a formula only has finitely many parameters. Fixing this convention, it will follow from the results of this section that $n$-$\eth$-dividing of a formula $\varphi(x, b)$ is independent of adding or removing unused parameters in $b$ for $n > 1$; this is not known for $n =1$.}

(2b)  A formula $\varphi(x, b)$ \emph{$(n+1)$-$\eth$-forks} over a model $M$ if there are $\varphi_{i}(x, b_{i})$ $(n+1)$-$\eth$-dividing over $M$ so that $\models \varphi(x, b) \rightarrow \bigvee_{i=1}^{n} \varphi_{i}(x, b_{i})$.

(2c) We say that $a$ is \emph{$(n+1)$-$\eth$-independent} from $b$ over $M$, denoted $a \ind_{M}^{\eth^{n+1}} b$, if $\mathrm{tp}(a/Mb)$ contains no formulas $(n+1)$-$\eth$-forking over $M$.

\end{definition}

It will be useful for our main results to show that $n$-$\eth$-forking coincides with $n$-$\eth$-dividing in general, for $n > 1$; this is not known to be the case for $n=1$, so this case will need to be handled separately in proving our main results.

\begin{lemma}\label{5-ext}
(1) The relation $\ind^{\eth^n}$ has right-extension for $n \geq 0$: if $a \ind_{M}^{\eth^n} b$ then for any $c$ there is $a \equiv_{Mb} a'$ with $a' \ind_{M}^{\eth^n} bc$.

(2) The relation $\ind^{\eth^n}$ has left-extension for $n \geq 1$: if $a \ind_{M}^{\eth^n} b$ then for any $c$ there is $c \equiv_{Ma} c'$ with $ac' \ind_{M}^{\eth^n} b$.

\end{lemma}

\begin{proof}
(1) This is known for $n=0$ and follows as in that case in the standard way for $n \geq 1.$ Suppose that  $a \ind_{M}^{\eth^n} b$, but there were no such $a'$. Then by compactness, there would be some formulas $\varphi(x, b) \in \mathrm{tp}(a/Mb)$, and $\varphi_{i}(x, d_{i})$ $n$-$\eth$-forking over $M$ for $d_{i} \subseteq Mbc$, so that $\models \varphi(x, b) \rightarrow \bigvee^{n}_{i=1} \varphi_{i}(x, d_{i})$. By definition of $n$-$\eth$-forking, for $1 \leq i \leq n$, there are $\varphi_{ij}(x, b_{ij})$ $n$-$\eth$-dividing over $M$ so that $\models \varphi_{i}(x, d_{i}) \rightarrow \bigvee_{j =1}^{n_{i}} \varphi_{ij}(x, b_{ij})$. Then $\models \varphi(x, b) \rightarrow \bigvee_{i=1}^{n} \bigvee_{j=1}^{n_{i}}\varphi_{ij}(x, b_{ij})$, so $\varphi(x, b)$ $n$-$\eth$-forks over $M$, contradicting $a \ind_{M}^{\eth^n} b$.

(2) Suppose that  $a \ind_{M}^{\eth^n} b$. Let $M'\succ M$ be an $(|M|+|T|)^{+}$-saturated elementary extension of $M$. By (1) there is $a' \equiv_{Mb} a$ with $a' \ind_{M}^{\eth^n} M'$. So by replacing $a$ with $a'$, we can assume that $b = M'$ is an $(|M|+|T|)^{+}$-saturated elementary extension of $M$.

We next show that, for any $d$ with $\mathrm{tp}(d/M')$ containing no formulas $n$-$\eth$-dividing over $M$, $d \ind_{M}^{\eth^n} M'$. This argument is standard from the literature. Suppose otherwise. Then $\mathrm{tp}(d/M')$ contains a $\varphi(x, b)$ $n$-$\eth$-forking over $M$ for $b \subset M'$. So there are $\varphi(x, b_{i})$ $n$-$\eth$-dividing over $M$ with $\models \varphi(x, b) \rightarrow \bigvee_{i=1}^{n} \varphi_{i}(x, b_{i})$. By $(|M|+|T|)^{+}$-saturation of $M'$ there are $b'_{1}, \ldots, b'_{n} \subset M'$ with $b'_{1} \ldots b'_{n} \equiv_{Mb} b_{1} \ldots b_{n}$. So $\varphi(x, b'_{i})$ $n$-$\eth$-divide over $M$ for $1 \leq i \leq n$ and $\models \varphi(x, b) \rightarrow \bigvee_{i=1}^{n} \varphi_{i}(x, b'_{i})$. By the latter, that $\varphi(x, b) \in \mathrm{tp}(d/M')$ and $\mathrm{tp}(d/M')$ is a complete type over $M'$ implies that there is some $1 \leq i \leq n$ with $\varphi_{i}(x, b'_{i}) \in \mathrm{tp}(d/M')$, a contradiction.

Now consider any $c$. It suffices to find $c'a' \equiv_{M} ca$ with $a' \equiv_{M'} a$ and $\mathrm{tp}(c'a'/M')$ containing no formulas $n$-$\eth$-dividing over $m$. So by compactness, for $\psi(y, x) \in \mathrm{tp}(ca/M)$ and $\varphi(x, d) \in \mathrm{tp}(a/M')$ with $d \subseteq M'$, it suffices to find $c'a'$ with $\models \psi(c', a') \wedge \varphi(a', d)$ so that $\mathrm{tp}(c'a'/Md)$ contains no formulas $\varphi'(y, x, e)$ $n$-$\eth$-dividing over $M$ with $e \subseteq d$. The formula $\exists y (y, x) \wedge \varphi(x, d)$ belongs to $\mathrm{tp}(a/M')$. So by $a \ind_{M}^{\eth^n} M'$, $\exists y (y, x) \wedge \varphi(x, d)$ does not $n$-$\eth$-divide over $M$. By definition of $n$-$\eth$-dividing, there is an $\ind^{\eth^{n-1}}$-Morley sequence over $M$, $I=\{d_{i}\}_{i < \omega}$ with $d_{0} = d$, so that $\{\exists y (y, x) \wedge \varphi(x, d_{i})\}_{i < \omega}$ is consistent, realized by $a'$. By Ramsey's theorem, compactness, and an automorphism, we choose $a'$ so that $I$ is indiscernible over $Ma'$. In particular, $\models \exists y \psi(y,a') $, so choose $c'$ so that $\models \psi(c', a')$. By another application of Ramsey's theorem, compactness and an automorphism, we can choose $c'$ so that $I$ is indiscernible over $Ma'c'$. It remains to show that $\mathrm{tp}(c'a'/Md)$ contains no formulas $\varphi'(y, x, e)$ $n$-$\eth$-dividing over $M$ with parameters $e \subseteq d$. For $i < \omega$ there are $e_{i} \subseteq d_{i}$ with $\{e_{i}\}_{i < \omega}$ $Ma'c'$-indiscernible and $e_{0} =e$. By definition of $\ind^{\eth^{n-1}}$-Morley sequence, for $i < \omega$, $d_{i} \ind_{M}^{\eth^{n-1}} d_{0} \ldots d_{i-1}$. So it follows from the definition of $\ind_{M}^{\eth^{n-1}}$ that $e_{i} \ind_{M}^{\eth^{n-1}} e_{0} \ldots e_{i-1}$ (i.e. $\ind_{M}^{\eth^{n-1}}$ is monotone.) So $\{e_{i}\}_{i < \omega}$ is an $\ind_{M}^{\eth^{n-1}}$-Morley sequence over $M$. Let $\varphi'(y, x, e) \in \mathrm{tp}(c'a'/Md)$. Then by $Ma'c'$-indiscernibility, $\{\varphi'(y, x, e_{i})\}_{i < \omega}$ is consistent, realized by $a'c'$. So $\varphi'(y, x, e)$ does not $n$-$\eth$-divide over $M$.

\end{proof}

\begin{proposition}\label{5-fd}
For $n \geq 2$, $n$-$\eth$-forking coincides with $n$-$\eth$-dividing.

\end{proposition}

\begin{proof}
    Exactly as in Fact 6.1 of \cite{GFA}, using right- and left-extension for $\ind^{\eth^{n-1}}$, and the standard arguments. Suppose $\varphi(x, b)$ $n$-$\eth$-forks over $M$. Then $\varphi(x, b) \to \bigvee_{j =1}^{N} \varphi_{j}(x, c^{j})$ for some $\varphi_{j}(x, c^{j})$ $n$-$\eth$-dividing over $M$. We show that $\varphi(x, b)$ $n$-$\eth$-divides over $M$; suppose otherwise. Then by the definition of $n$-$\eth$-dividing and compactness, there is an $\ind^{\eth^{n-1}}$-Morley sequence $\{b_{i}\}_{i < \kappa}$, for large $\kappa$, i.e. indiscernible over $M$ with $b_{i} \ind^{\eth^{n-1}} b_{<i}$ for $i < \kappa$, with $b_{0} = b$ and $\{\varphi(x, b_{i})\}_{i< \kappa}$ consistent. By induction we find $c^{j}_{i}$, $1 \leq j \leq N$, $i< \kappa$, so that $\{c^{j}_{i}\}^{N}_{j = 1}b_{i} \equiv_{M} \{c^{j}\}^{N}_{j =1}b$ and $\{c^{j}_{i}\}^{N}_{j = 1}b_{i} \ind_{M}^{\eth^{n-1}}  \{c^{j}_{<i}\}^{N}_{j = 1}b_{<i}$ for $i < \kappa $. Suppose by induction that for $\lambda < \kappa$ we have found $c^{j}_{i}$, $1 \leq j \leq N$, $i< \lambda$, so that $\{c^{j}_{i}\}^{N}_{j = 1}b_{i} \equiv_{M} \{c^{j}\}^{N}_{j =1}b$ and $\{c^{j}_{i}\}^{N}_{j = 1}b_{i} \ind_{M}^{\eth^{n-1}}  \{c^{j}_{<i}\}^{N}_{j = 1}b_{<i}$ for $i < \lambda $. Then because $b_{\lambda}  \ind_{M}^{\eth^{n-1}} b_{<\lambda}$, by right extension we could have chosen $c^{j}_{i}$, $1 \leq j \leq N$, $i< \lambda$ so that $b_{\lambda} \ind_{M}^{\eth^{n-1}}  \{c^{j}_{<\lambda}\}^{N}_{j = 1}b_{<\lambda}$. Now by left extension and an automorphism, find $c^{j}_{\lambda}$, $1 \leq j \leq N$, with $\{c_{\lambda}^{j}\}^{N}_{j =1}b_{\lambda} \equiv_{M} \{c^{j}\}^{n}_{j =1}b$ and $\{c_{\lambda}^{j}\}^{N}_{j =1}b_{\lambda} \ind_{M}^{\eth^{n-1}}  \{c^{j}_{<\lambda}\}^{n}_{j = 1}b_{<\lambda}$. This completes the induction. Now by the Erdős-Rado theorem and an automorphism, we can find $c'^{j}_{i}$ for $i < \omega$, $1 \leq j \leq N$, so that $\{c_{i}'^{1} \ldots c'^{N}_{i}b_{i}\}_{i< \omega}$ is an $\ind^{\eth^{n-1}}$-Morley sequence over $M$ with $c_{0}'^{1} \ldots c'^{N}_{0}b_{0}=c^{1} \ldots c^{N}b$.  Now we give the standard argument to get a contradiction. Each $\{c'^{j}_{i}\}_{i < \omega}$ for $1 \leq j \leq N$ is an $\ind^{\eth^{n-1}}$-Morley sequence over $M$ with $c'^{j}_{0} = c^{j}$. Let $a$ realize $\{\varphi(x, b_{i})\}_{i< \kappa}$.  For each $i < \omega$, $\models \varphi(x, b_{i}) \rightarrow \bigvee_{j =1}^{N} \varphi_{j}(x, c_{i}'^{j})$, so $\models \varphi_{j}(a, c_{i}'^{j})$ for some $1 \leq j \leq n$. So there is some $1 \leq j \leq n$ so that $\models \varphi_{j}(a, c_{i}'^{j})$ for infinitely many $i < \omega$. By an automorphism, $\{\varphi_{j}(x, c'^{j}_{i})\}_{i < j}$ is consistent. But because $\{c'^{j}_{i}\}_{i < \omega}$ is an $\ind^{\eth^{n-1}}$-Morley sequence over $M$ with $c'^{j}_{0} = c^{j}$, this contradicts $n$-$\eth$-dividing of $\varphi_{j}(x, c^{j})$ over $M$.
\end{proof}

While we will not use the results of the rest of this section in the sequel, they will give some additional motivation to the main results of this paper and to the open question posed at the end. We will be interested in the question of when $\ind^{\eth^{n}}$ has properties analogous to Kim-independence in $\mathrm{NSOP}_{1}$, including transitivity (\cite{KR19}) and especially symmetry (\cite{KR17}). If the answer to Question \ref{5-openquestion} is positive, either of these properties will be equivalent to $\mathrm{NSOP}_{2^{n+1}+1}$, so in particular symmetry of $\ind^{\eth^{n}}$ will imply symmetry of $\ind^{\eth^{n+1}}$. However, even the question of whether symmetry of $\ind^{\eth^{n}}$ implies symmetry of $\ind^{\eth^{n+1}}$ is still open. But if we strengthen symmetry to an analogue of the stable forking conjecture, a linear hierarchy starts to emerge. 

\begin{definition}
    A theory $T$ satisfies the \emph{stable $n$-$\eth$-forking conjecture} if whenever $a \nind^{\eth^{n}}_{M} b$, there is some $L(M)$-formula $\varphi(x, y)$, which is stable as an $L(M)$-formula, so that $\models \varphi(a, b)$ and $\varphi(x, b)$ $n$-$\eth$-forks over $M$.
\end{definition}

\begin{remark}\label{5-sfc}
    Unlike the stable forking conjecture for simple theories, here we require only that $\varphi(x, y)$ be stable as a $L(M)$-formula, not an $L$-formula. It makes sense to make this allowance outside of the simple case, as the analogous conjecture about Kim-forking fails for, say, $T^{\mathrm{feq}}$ (see, e.g., \cite{She94}, \cite{SD00}, \cite{CR15}, \cite{KR17}), when the stability is as an $L$-formula, but the ``stable Kim-forking conjecture" is open for $\mathrm{NSOP}_{1}$ theoires when the stability is taken to be as an $L(M)$ formula.
\end{remark}

Note that if a stable formula $1$-$\eth$-forks over $M$, it forks over $M$, so by basic stability theory divides with respect to any nonforking Morley sequence; that is, $1$-$\eth$-divides over $M$. If $T$ satisfies the stable $1$-$\eth$-forking conjecture, it follows that if $\mathrm{tp}(a/Mb)$ contains no $1$-$\eth$-dividing formulas, $a \ind_{M}^{\eth^{1}} b$.

The first part of this fact is well-known; the second is observed without proof for theories in \cite{Bu91}, but the evident proof works equally well for formulas.

\begin{fact}
    If a formula $\varphi(x, y)$ is stable, it is without the tree property and is low: there is some $k$ so that for $\{b_{i}\}_{i \in I}$ an indiscernible sequence, $\{\varphi(x, b_{i})\}_{i \in I}$ is inconsistent if and only if it is $k$-inconsistent.
\end{fact}

The argument for the following is similar to the literature: see, for example, Theorem 5.16 of \cite{KR17} or Theorem 5.3 of \cite{NSOP2}. The construction of the tree is lifted mostly word-for-word from the proof of Theorem 5.3 of \cite{NSOP2}, but we follow the local approach in the paragraph following the proof of that theorem.

\begin{proposition}
    If $T$ satisfies the stable $n$-$\eth$-forking conjecture, then $\ind^{\eth^{n}}$ is symmetric.
\end{proposition}

\begin{proof} 

Assume $a \ind_{M}^{\eth^{n}}b$ but $b \nind_{M}^{\eth^{n}}a$. Let $\varphi(x, y)$ be a stable $L(M)$-formula with $\models \varphi(b, a)$ and $\varphi(x, a)$ a stable formula $n$-$\eth$-forking over $M$; by Proposition \ref{5-fd} and the remarks after Remark \ref{5-sfc}, this $n$-$\eth$-divides over $M$. We find, for any $n$, a tree $(I_{n}, J_{n})= (\{a_{\eta}\}_{\eta \in \omega^{\leq n}}, \{b_{\sigma}\}_{\sigma \in \omega^{n}})$, infinitely branching at the first $n+1$ levels and then with each $a_{\sigma}$ for $\sigma \in \omega^{n}$ at level $n+1$ followed by a single additional leaf $b_{\sigma}$ at level $n+2$, satisfying the following properties:

(1) For $\eta \unlhd \sigma$, $\models \varphi(b_{\sigma}, a_{\eta})$.

(2) For $\eta \in \omega^{< n}$, the branches at $\eta$ form an $\ind^{\eth^{n-1}}$-Morley sequence over $M$ indiscernible over $Ma_{\eta}$, so by Proposition \ref{5-fd} and the remarks after Remark \ref{5-sfc}, $a_{\eta}$ is $n$-$\eth$-independent over $M$ from those branches taken together.

Suppose $(I_{n}, J_{n})$ already constructed; we construct $(I_{n+1}, J_{n+1})$. We see that the root $a_{\emptyset}$ of $(I_{n}, J_{n})$ is $n$-$\eth$-independent from the rest of the tree, $(I_{n}, J_{n})^{*}$: for $n = 0$ this is just the assumption $a \ind^{\eth^{n}}_{M} b$, while for $n > 0$ this is (2). So by extension we find $a_{\emptyset}' \equiv_{M(I_{n}J_{n})^{*}} a_{\emptyset}$ (so guaranteeing (1)), to be the root of $(I_{n+1}, J_{n+1})$, with $a'_{\emptyset} \ind^{K^{*}}_{M} I_{n}J_{n}$. Then by applying $a \ind^{K^{*}}_{M} I_{n}J_{n}$ to the formulas giving (1), find some $\ind^{\eth^{n-1}}$-Morley sequence $\{(I_{n},J_{n})^{i}\}_{i \in \omega}$ with $(I_{n},J_{n}) \equiv_{M} (I_{n},J_{n})^{i}$ indiscernible over $Ma'_{\emptyset}$, guaranteeing (2) and preserving (1), and reindex accordingly.

By lowness and $n$-$\eth$-dividing of $\varphi(x, a)$ over $M$, the successors to each node witness $k$-dividing of $\varphi(x, a)$ over $M$ for some fixed $k$. This together with $1$ gives the $k$-tree property for $\varphi(x, a)$, a contradiction.
\end{proof}

We now show the linear hierarchy obtained when symmetry of $n$-$\eth$-independence is improved to the conclusion of the stable $n$-$\eth$-forking conjecture.

\begin{proposition}
    If $T$ satisfies the stable $n$-$\eth$ forking conjecture, then $\ind^{\eth^{n}}= \ind^{\eth^{n+1}}$ (so $\ind^{\eth^{n}}= \ind^{\eth^{m}}$ for $m \geq n$.)
\end{proposition}

\begin{proof}
    It suffices to show that for $\varphi(x, y)$ a stable formula, if $\varphi(x, b)$ $n$-$\eth$-forks over $M$, so $n$-$\eth$-divides over $M$, then it $n+1$-$\eth$-divides over $M$. Let $\{b_{i}\}_{i < \omega}$ be an $\ind^{\eth^{n}}$-Morley sequence over $M$; by the definition of $n+1$-$\eth$-dividing, it suffices to show that $\{\varphi(x, b_{i})\}_{i < \omega}$ is $k$-inconsistent. Suppose otherwise. By compactness, extend $I=\{b_{i}\}_{i < \omega}$ to a $\ind^{\eth^{n}}$-Morley sequence $\{b_{i}\}_{i < \omega+\omega}$ over $M$. Then $\{b_{i}\}_{\omega \leq i < \omega + \omega}$ is a (nonforking) Morley sequence over $M I$ that does not witness dividing of $\varphi(x, b_{\omega})$. So it suffices to show that $\varphi(x, b_{\omega})$ divides over $MI$ anyway, contradicting the basic properties of stability. By the previous proposition, $\ind^{\eth^{n}}$ is symmetric, so $I \ind_{M}^{\eth^{n}} b_{\omega}$. Note that $\varphi(x, b_{\eta})$ divides over $M$. So by lowness, there is some $k$ so that each $\ind^{\eth^{n-1}}$-Morley sequence over $M$ starting with $b_{\omega}$ witnesses $k$-dividing of $\varphi(x, b_{\eta})$. Now for any formula in $\mathrm{tp}(b_{\omega}/MI)$, by $I \ind_{M}^{\eth^{n}} b_{\omega}$ there is an $\ind^{\eth^{n-1}}$-Morley sequence over $M$ starting with $b_{\omega}$, each term of which realizes this formula. In sum, for any formula in $\mathrm{tp}(b_{\omega}/MI)$, there is an $M$-indiscernible sequence of realizations of this formula witnessing the $k$-dividing of $\varphi(x, b_{\omega})$ over $M$. So by compactness, $\varphi(x, b_{\eta})$ $k$-divides over $MI$.
\end{proof}

In the next section, we will characterize $\ind^{\eth^{n}}$ in the classical examples of $\mathrm{NSOP}_{2^{n+1}+1}$ theories, for $n \geq 1$; it will be trivial in these examples, so satisfies the stable $n$-$\eth$-forking conjecture. Note that, if we start with the analogous stability assumption for Conant-independence (see \cite{GFA}), the proof of the previous propositions show that it coincides with $n$-$\eth$-independence for $n \geq 1$, and is symmetric. There are no known counterexamples to the ``stable Conant-forking conjecture" for $\mathrm{NSOP}_{4}$ theories. \footnote{Conant-independence is characterized for some classical examples of $\mathrm{NSOP}_{4}$ theories in \cite{GFA} For the Fraïssé-Hrushovski constructions of finite language, where the author has shown that Conant-independence coincides with $d$-independence, the proof of the stable forking conjecture for the simple case of these structures in \cite{Po03}, \cite{Ev02} should extend to the general case using this characterization.}

\section{Attainability/examples}
In $\mathrm{NSOP}_{1}$ theories, $\ind^{\eth^{n}}$ is just Kim-independence for $n \geq 1$. Moreover, in the examples of $\mathrm{NSOP}_{4}$ theories where Conant-independence has been characterized, it coincides with $\ind^{\eth^{n}}$ and is symmetric (see the end of the previous section). We now give some proper examples. These examples will show the attainability of $\mathrm{SOP}_{2^{n+1}}$ as the bound on the levels of the $\mathrm{SOP}_{k}$ hierarchy where $\ind^{\eth^{n}}$ can be symmetric or transitive.

\begin{example}\label{5-frcg} (Free roots of the complete graph) In \cite{CW04}, Casanovas and Wagner show that the theory $T_{n}^{-}$ of metric spaces valued in the set $\{0, \ldots, n\}$ has a model companion $T_{n}$. (More precisely, this is interdefinable with the theory introduced in \cite{CW04}, but we use the language of metric spaces.) They show that this theory is $\omega$-categorical, eliminates quantifiers, and has trivial algebraicity, and that it is $\mathrm{NSOP}$ but not simple. Later, Conant and Terry show in \cite{CT16} that $T_{n}$ is strictly $\mathrm{NSOP}_{n+1}$. We want to show the following
\begin{theorem}\label{5-existence}
In $T_{n}$, if $2^{k+1} \leq n$, $a \ind^{\eth^{k}}_{M} b$ if and only if $a \ind^{a}_{M} b =: a \cap b = M$. Therefore, there are $\mathrm{SOP}_{2^{k+1}}$ theories where $\ind^{\eth^{k}}$ is symmetric.

\end{theorem}

We first show the following lemma:

\begin{lemma}\label{5-fa}
    Let $C, \subseteq A, B$ be metric spaces valued in $\{0, \ldots, n\}$. Then there is a metric space $D$ valued in $\{0, \ldots, n\}$ together with isometric embeddings $\iota_{A}: A \hookrightarrow D$ and $\iota_{B}: B \hookrightarrow D$ with $\iota_{A}|_{C} = \iota_{B}|_{C}$ and with, for $a \in A \backslash C$, $b \in B \backslash C$, and $d_{ab} =: d_{D}(\iota_{A}(a), \iota_{B}(b))$

   (a) $d_{ab} = m_{ab} =: \mathrm{min}_{c \in C}(d_{A}(a, c) + d_{B}(b, c))$ if $m_{ab} < n^{*} = :\lceil \frac{n}{2} \rceil$.

   (b) $d_{ab} = m^{ab} =: \mathrm{max}_{c \in C}(|d_{A}(a, c) - d_{B}(b, c)|)$ if  $m^{ab} > n^{*}$

   (c) Otherwise, $d_{ab} = n^{*}$.
\end{lemma}

\begin{proof}
    We may assume $A \cap B = C$ as sets. So it suffices to define a metric on $D= A \cup B$ extending that on $A$ and $B$ and satisfying (a), (b), (c). We claim that for all $a \in A \backslash C$, $b \in B \backslash C$, $m_{ab} \geq m^{ab}$, so only one of (a), (b), (c) may hold. Suppose otherwise. Then there are $c_{*}, c^{*} \in C$ with $d(a, c_{*}) + d(b, c_{*}) < |d(a, c^{*}) - d(b, c^{*})|$. But because $T_{n}$ has quantifier elimination and trivial algebraicity, the class of finite metric spaces valued in $\{0, \ldots, n\}$ has the strong amalgamation property, so there is a metric $d: \{a, b, c_{*}, c^{*}\}^{2} \to \{0 \ldots n\}$ extending the metric on $\{a, c_{*}, c^{*}\}$ and $\{b, c_{*}, c^{*}\}$. Then $d(a, b) \leq d(a, c_{*}) + d(b, c_{*}) < |d(a, c^{*}) - d(b, c^{*})| \leq d(a, b)$, a contradiction. So conditions (a), (b) and (c), together with the requirement of extending the metric on $A$ and $B$, give a well-defined function $d: D^{2} \to \{0 \ldots n\}$, and it remains to show this is a metric.

    Suppose first that $a  \in A \backslash C$, $b \in B \backslash C$, and $c \in C$. Then $d$ satisfies the triangle inequality on $\{a, b, c\}$ if $|d(a, c) - d(b, c)| \leq d(a, b) \leq d(a, c) + d(b, c)$. The first inequality is by (b), (c) and the second is by (a), (c). 

    Now suppose without loss of generality that $a_{1}, a_{2} \in A \backslash C $ and $b \in B \backslash C$. It remains to show the triangle inequality on $\{a_{1}, a_{2}, b\}$. By the definition of $n^{*}$, this is immediately the case if $d(a_{1}, b)= d(a_{2}, b) = n^{*}$. Otherwise, without loss of generality there are three cases, where $d(a_{1}, b) < n^{*}$ and $d(a_{2}, b) > n^{*}$, where $d(a_{1}, b) < n^{*}$ and $d(a_{2}, b) \leq n^{*}$, and where $d(a_{1}, b) > n^{*}$ and $d(a_{2}, b) \geq n^{*}$.

    In the case where $d(a_{1}, b) < n^{*}$ and $d(a_{2}, b) > n^{*}$, easily $d(a_{1}, b) \leq d(a_{2}, b) + d(a_{1}, a_{2})$. To show $d(a_{2}, b) \leq d(a_{1}, a_{2}) + d(a_{1}, b)$, by the strong amalgamation property, there is some metric $d'$ on $D$ extending $d$ on $A$ and $B$. So

    $$d(a_{2}, b) \leq d'(a_{2}, b) \leq d'(a_{1}, b) + d'(a_{1}, a_{2}) = d'(a_{1}, b) +d(a_{1}, a_{2}) \leq d(a_{1}, b) + d(a_{1}, a_{2})$$ Note that (b) is used for the first inequality, and (a) is used for the last inequality. Finally, we show $d(a_{1}, a_{2}) \leq d(a_{2}, b) + d(a_{1}, b) $, so $d(a_{2}, b) \geq d(a_{1}, a_{2}) -d(a_{1}, b)$. Note that $d(a_{1}, b) = d(a_{1}, c) + d(b, c) $ for some $c \in C$. Then by (b), $d(a_{2}, b) \geq d(a_{2}, c) - d(c, b) \geq (d(a_{1}, a_{2}) - d(a_{1}, c)) - d(c, b) =d(a_{1},a_{2}) -d(a_{1}, b) $.

    In the case where $d(a_{1}, b) < n^{*}$ and $d(a_{2}, b) \leq n^{*}$, we first show $d(a_{1}, b) \leq d(a_{2}, b) + d(a_{1}, a_{2})$ and $d(a_{2}, b) \leq d(a_{1}, b) + d(a_{1}, a_{2})$. If additionally, $d(a_{2}, b) = n^{*}$, then the first of these inequalities is immediate, so it suffices to prove the second inequality, as by symmetry we will have then proven the first inequality when both $d(a_{1}, b) < n^{*}$ and $d(a_{2}, b) < n^{*}$. By (a), there is some $c \in C$ with $d(a_{1}, b) = d(a_{1}, c) + d(b, c) $, so by (a), (c), $d(a_{2}, b) \leq d(a_{2}, c) + d(c, b) \leq d(a_{1}, a_{2}) + d(a_{1}, c) + d(c, b) = d(a_{1}, a_{2}) + d(a_{1}, b)$. Finally, we show in this case that $d(a_{1}, a_{2}) \leq d(a_{2}, b) + d(a_{1}, b)$. For any $c \in C$, if $d(a_{2}, b) = n^{*}$ this must be because $m^{a_{2}b} \leq n^{*}$, and if $d(a_{2}, b) < n^{*}$, then $d(a_{2}, b) = m_{ab} \geq m^{ab}$, so $d(a_{2}, b) \geq d(a_{2}, c) - d(c, b)$. So $d(a_{1}, a_{2}) \leq d(a_{2}, b) + d(a_{1}, b)$ follows exactly as in the case of $d(a_{1}, b) < n^{*}$ and $d(a_{2}, b) > n^{*}$.

    Finally, if $d(a_{1}, b) > n^{*}$ and  $d(a_{2}, b) \geq n^{*}$, then $d(a_{1}, a_{2}) \leq d(a_{1}, b)+ d(a_{2}, b)$ is immediate. We next show $d(a_{1}, b) \leq d(a_{2}, b) + d(a_{1}, a_{2})$ and $d(a_{2}, b) \leq d(a_{1}, b) + d(a_{1}, a_{2})$. If $d(a_{2}, b) = n^{*}$, then the second inequality is immediate, so it suffices to prove the first inequality, as by symmetry we will have then proven the second inequality when both $d(a_{1}, b) > n^{*}$ and $d(a_{2}, b) > n^{*}$. By (a) there will be some $c \in C$ with either $d(a_{1}, b) = d(a_{1}, c) -d(b, c) $ or $d(a_{1}, b) = d(b, c) - d(a_{1}, c)$. In the case of $d(a_{1}, b) = d(a_{1}, c) -d(b, c) $ by (b), (c), $d(a_{2}, b) \geq d(a_{2}, c) - d(b, c) \geq d(a_{1}, c) - d(a_{2}, a_{1}) - d(b, c)= d(a_{1}, b) - d(a_{2}, a_{1}) $, proving the first inequality in this case. In the case of $d(a_{1}, b) = d(b, c) - d(a_{1}, c)$, again by (b), (c), $d(a_{2}, b) \geq  d(b, c) - d(a_{2}, c) = d(a_{1}, c) + d(a_{1}, b) - d(a_{2}, c) = d(a_{1}, b) -(d(a_{2}, c) - d(a_{1}, c)) \geq d(a_{1}, b) - d(a_{1}, a_{2})$, proving the first inequality in the other case.

\end{proof}

\begin{remark}\label{5-nstar}
Say that in the statement of Lemma \ref{5-fa}, we let $n^{*} = \lceil \frac{n}{2} \rceil+1$ instead of $\lceil \frac{n}{2} \rceil$. Then the above proof works: $n^{*}$ can be any constant at least $\lceil \frac{n}{2} \rceil$, and the only place that $n^{*} \geq \lceil \frac{n}{2} \rceil$ is used is in the case where $a_{1}, a_{2} \in A \backslash C $, $b \in B \backslash C$ and  $d(a_{1}, b)= d(a_{2}, b) = n^{*}$. In the sequel, we will still let $n^{*}$ denote $\lceil \frac{n}{2} \rceil$

\end{remark}

\begin{definition}
Let $C \subseteq A, B$ be subspaces of some fixed metric space with values in $\{0, \ldots, n^{*}\}$, with $A \cap B = C$.

(1) $A$ and $B$ are \emph{freely amalgamated} over $C$ if the inclusions $\iota_{A}$ and $\iota_{B}$ satisfy the conclusion of Lemma \ref{5-fa}.

(2a) For $k \leq n$, $A$ and $B$ have \emph{distance $\leq n$} over $C$ if for $a \in A \backslash C$, $b \in B \backslash C$, $d(a, b) \leq \mathrm{max}(k, m^{ab})$.

(2b) For $ k \geq 1$, $A$ and $B$ have \emph{distance $\geq n$} over $C$ if for $a \in A \backslash C$, $b \in B \backslash C$, $d(a, b) \geq \mathrm{max}(k, m_{ab})$.

\end{definition}

\begin{lemma}\label{5-distance}
    Let $C \subseteq A, B, D$ be subsets of a fixed metric space, and $1 \leq k_{1} \leq n_{*} \leq n_{2} \leq n$. Suppose $A \cup D$ and $B \cup D$ are freely amalgamated over $D$, and both $A$ and $B$ have distance $\geq k_{1}$ and $\leq k_{2}$ from $D$ over $C$. Then $A$ has distance $\geq \mathrm{min}(n^{*}, 2k_{1})$ and $ \leq \mathrm{max}(k_{2} - k_{1}, n^{*})$ from $B$ over $C$, and moreover, $D$ has distance $\geq k_{1}$ and $\leq k_{2}$ from $A \cup B$ over $C$.
\end{lemma}

\begin{proof}
    The second clause is obvious, so we prove the first; let $d$ be the metric. First we show that $A$ has distance $\geq \mathrm{min}(n^{*}, 2k_{1})$ from $B$ over $C$. Let $a \in A \backslash C$, $b \in B\backslash C$. It suffices to show that if $d(a, b)< \mathrm{min}(n^{*}, 2k_{1})$ then there is some $c \in C$ with $d(a, c) + d(c, b) \leq d(a, b)$. Because $d(a, b) < n^{*}$, there is some $d \in D$ with $d(a, d)+ d(b, d) =d(a, b) < 2k_{1}$. So either $d(a, d)$ or $d(b, d)$ must be less that $k_{1}$. Without loss of generality,  $d(a, d) < k_{1}$. Then because $D$ and $A$ have distance $\leq k_{1}$ over $C$, there is some $c \in C$ with $d(a, d) = d(a, c) + d(c, d)$. Then $d(a, b) = d(a, c) + d(c, d) + d(b, d) \geq d(a, c) + d(c, b)$. Next we show that $A$ has distance  $ \leq \mathrm{max}(k_{1} - k_{2}, n^{*})$ from $B$ over $C$.  Let $a \in A \backslash C$, $b \in B\backslash C$. It suffices to show that if $d(a, b) > \mathrm{max}(k_{2} - k_{1}, n^{*})$, there is some $c \in C$ with  $d(a, b) \leq |d(a, c) - d(b, c)|$. Because $d(a, b) > n^{*}$, without loss of generality there is some $d \in D$ with $d(a, d) - d(b, d) =d(a, b) > k_{2} - k_{1}$. So either $d(a, d) > k_{2}$ or $d(b, d) < k_{1}$. In the case where $d(a, d) > k_{2}$, since $A$ has distance $\geq k_{2}$ from $D$ over $C$, there is some $c \in C$ so that either $d(a, d) = d(a, c) - d(c,d)$ or $d(a, d) = d(c, d) - d(a, c)$. If $d(a, d) = d(a, c) - d(c,d)$, then $d(a, b) = d(a, c) - d(b,d) - d(c,d) \leq d(a, c) - d(b, c) $. If $d(a, d) = d(c, d) - d(a, c)$, then  $d(a, b) = d(c, d) - d(a, c) -d(b, d) \leq d(c, b) - d(a, c)$. In the case where $d(b, d) < k_{1}$, since $B$ has distance $\geq k_{1}$ from $D$ over $C$, there is some $c \in C$ with $d(b, d) =d(b, c) + d(c, d)$. Then $d(a, b) = d(a, d) - d(b, c) -d(c, d) \leq d(a, c) - d(b, c)$.
\end{proof}

\begin{lemma}\label{5-distanceforking}
Let $A$ and $B$ have distance $\geq n^{*}$ and $\leq n^{*} + 1$ over $C$. Then $A \ind_{C}^{f} B$.  
\end{lemma}

\begin{proof}
    Again, let $d$ be the ambient metric. We first show that $\mathrm{tp}(A/CB)$ does not contain any formulas dividing over $C$. Let $I=\{B_{i}\}_{i < \omega}$ be a $C$-indiscernible sequencene with $B_{0} = B$. We can find a function $d^{*}: (A \cup I)^{2} \to \{0, \ldots n\}$ so that the bijection from $AB$ to $AB_{i}$ given by enumeration is an isomorphism, and so that $d^{*}$ agrees with $d$ on $I$. If we show this is a metric, then by quantifier elimination, there is $I' \equiv_{CB} I$ indiscernible over $A$, showing that $\mathrm{tp}(A/CB)$ does not contain any formulas dividing over $C$ since $I$ was arbitrary. Without loss of generality, it suffices to show the triangle inequality for $d^{*}$ on $\{b_{0}, b_{1}, a\}$ for $b_{0} \in B_{0}$, $b_{1} \in B_{1}$, $a \in A$, $b_{0} \neq b_{1}$. Suppose first that $d^{*}(b_{0}, a), d^{*}(b_{1}, a) \in \{n^{*}, n^{*}+1\}$. Then since $d^{*}(b_{0}, b_{1}) \geq 1$, the triangle inequality is immediate in all directions. Otherwise, without loss of generality, $d^{*}(b_{0}, a)$ is either $n^{*}$ or $n^{*}+1$, and $d^{*}(b_{1}, a)$ is either greater than $n^{*}+1$ and equal to $\mathrm{max}_{c \in C}(|d(b_{1}, c)-d(a,c)|)$ or less than $n^{*}$ and equal to $\mathrm{min}_{c \in C}(d(b_{1}, c)+d(a, c))$. In either case, by Lemma \ref{5-fa} and Remark \ref{5-nstar} applied to $C \subseteq A, I$, there is a metric on $A \cup I$ that agrees with $d^{*}$ on $\{a, b_{0}, b_{1}\}$, so $d^{*}$ satisfies the triangle inequality on $\{a, b_{0}, b_{1}\}$.

To show $A \ind^{f}_{C} B$, we need the following claim:

    \begin{claim}\label{5-dividingtoforking}
Let the relation $A \ind_{C} B$ be defined to hold when $A$ and $B$ have distance $\geq n^{*}$ and $\leq n^{*} + 1$ over $C$ has right extension: if $B \subseteq D$ there is $A'\equiv_{B} A$ with $A' \ind_{C} D$.
    \end{claim}

    \begin{proof}
        By Lemma \ref{5-fa}, we may find $A' \equiv_{B} A$ so that $AB$ is freely amalgamated with $D$ over $B$. So it suffices to show that if $A$ is freely amalgamated with $D$ over $B$ and $A \ind_{C} B$, then $A \ind_{C} D$. We first show that $A$ and $D$ have a distance of $\geq n^{*}$ over $C$. Suppose $a \in A \backslash C$, $d \in D \backslash C$, and $d(a, d) < n^{*}$. Then there is some $b \in B$ (and we can assume $b \neq d$) so that $d(a, d) = d(b, d) + d(b, a)$ Then $d(b, a) < n^{*}$ so there is some $c \in C$ so that $d(b, a) = d(b, c) + d(a, c)$. So $d(a, d) = d(b, d) + d(b, c) + d(a, c) \geq d(d, c) + d(a, c) $, as desired. We now show that $A$ and $D$ have a distance of $\leq n^{*} + 1$ over $C$. Suppose $a \in A \backslash C$, $d \in D \backslash C$, and $d(a, d) > n^{*} + 1$. Then there is some $b \in B$ (and we can assume $b \neq d$) so that either $d(a, d)= d(b, d) - d(b, a)$ or $d(a, d) = d(b, a) - d(b, d)$. Assume first that $d(a, d)= d(b, d) - d(b, a)$. Then as $d(a, d) > n^{*} + 1$, $d(b, a) < n^{*}$, so there is some $c \in C$ with $d(b, a) = d(b, c) + d(c, a)$. Then $d(a, d)= d(b, d) - d(b, c)-d(c, a) \leq d(d, c) - d(c, a)$. Now assume $d(a, d) = d(b, a) - d(b, d)$. Then as $d(a, d) > n^{*} + 1$, $d(b, a) > n^{*} +1$, so there is some $c \in C$ with either $d(b, a) = d(b ,c) - d(a, c)$ or $d(b, a) = d(a, c) - d(b, c)$. If $d(b, a) = d(b ,c) - d(a, c)$, then $d(a, d) = (d(b ,c)  - d(b, d)) - d(a, c)\leq d(d, c)-d(a, c)$. If $d(b, a) = d(a, c) - d(b, c)$, then $d(a, d) = d(a, c) -(d(b, c) + d(b, d)) \leq d(a, c)-d(d, c) $. Either way, this is as desired. 
    \end{proof}

Then if $A$ and $B$ have distance $\geq n^{*}$ and $\leq n^{*} + 1$ over $C$, we can assume $B$ is an $|C|^{+}$-saturated model by the claim, so the it follows as in the second paragraph of the proof of Lemma \ref{5-ext}.2 and the fact that $\mathrm{tp}(A/BC)$ does not divide over $C$ that $A \ind^{f}_{C} B$.

\end{proof}

Now suppose $2^{k+1} \geq n$. First, $\ind^{\eth^{k}}$ always implies $\ind^{a}$, as $\ind^{\eth^{0}}$ implies $\ind^{a}$ and if $\ind^{\eth^{i}}$ implies $\ind^{a}$ it is seen that $\ind^{\eth^{i+1}}$ implies $\ind^{a}$. We now show that $\ind^{a}$ implies $\ind^{\eth^{k}}$. Let $\ind^{0} = \ind^{a}$, and for $m \geq 1$, let $A\ind_{C}^{m} B$ indicate that $A$ has distance $\geq \mathrm{min}(2^{m}, n^{*})$ and $\leq \mathrm{max}(n^{*}, n - (\sum_{i=0}^{m-1} 2^{i}))=\mathrm{max}(n^{*}, n-(2^{m} -1))$ from $B$ over $C$. Then by repeated applications of Lemma \ref{5-distance}, where $k_{1} = 2^{i}$ and $k_{2} =  n - (\sum_{j=0}^{i-1} 2^{j})=n-(2^{i} -1)$ for $i \geq 0$, if $A \ind^{i}_{C} B$ there is an $\ind^{i+1}$-Morley sequence $\{B_{j}\}_{j < \omega}$ over $C$ with $B_{0} =B$ indiscernible over $A$. Moreover, $\ind^{k}$ implies $\ind^{f} = \ind^{\eth^{0}}$ and has right, and therefore left extension by Lemma \ref{5-distanceforking} and Claim \ref{5-dividingtoforking}, as $2^{k} \geq n^{*}$. So by the proof of Proposition \ref{5-fd}, and the fact that if $A \ind^{k-1}_{C} B$ there is an $\ind^{k}$-Morley sequence (and therefore an $\ind^{\eth^{0}}$-Morley sequence) $\{B_{j}\}_{j < \omega}$ over $C$ with $B_{0} =B$ indiscernible over $A$, $\ind^{k-1}$ implies $\ind^{\eth^{1}}$. Suppose inductively that $\ind^{k-i}$ implies $\ind^{\eth^{i}}$ for $1 \leq i \leq k$. Then by Proposition \ref{5-fd} and the fact that if $A \ind^{k - (i+1)}_{C} B$ there is an $\ind^{k-i}$-Morley sequence (and therefore an $\ind^{\eth^{i}}$-Morley sequence) $\{B_{j}\}_{j < \omega}$ over $C$ with $B_{0} =B$ indiscernible over $A$, $\ind^{k-(i+1)}$ implies $\ind^{\eth^{i+1}}$. So $\ind^{a} = \ind^{0}$ implies and is therefore equal to $\ind^{\eth^{k}}$.

This concludes the proof of Theorem \ref{5-existence}, and our discussion of the free roots of the complete graph.

\end{example}

\begin{example}\label{5-cyclefree}
    (Model companion of directed graphs with no directed cycles of length $\leq n$.) In Example 2.8.3 of \cite{She95}, Shelah shows that the theory of directed graphs with no directed cycles of length $\leq n$ has a strictly $\mathrm{NSOP}_{n+1}$ model companion $T_{n}$. (Note that if there is a directed cycle of length $\leq n$ that repeats vertices, there is one that does not repeat any vertices.) We show again that if $2^{k} \geq n$, then $\ind^{\eth^{k}}$ coincides with $\ind^{a}$. This will be for somewhat different reasons than Example \ref{5-frcg}, with successive approximations of forking-independence given by longer directed distances, rather than distances tending away from the extremes. We will use the same technique as in Theorem \ref{5-existence}, but the proof will be more intuitive.

    Shelah observes that $T_{n}$ has quantifier elimination in the language with binary relation symbols $R_{n}(x, y)$ indicating a path of length $n$ from $x$ to $y$. We refine this quantifier elimination. Let $n^{*} = \lceil \frac{n}{2} \rceil$. Then there must be a path of length $\leq n^{*}$ between any two nodes by existential closedness: if there is not one, one can be added on without creating any cycles of length $\leq n$. Moreover, by the lack of cycles, if there is a path of $\leq \frac{n}{2}$ there must be such a path in only one direction, by the absence of directed cycles of length $\leq n$. So if $n$ is even any pair of distinct nodes $a, b$ has one of $2n^{*}=n$ types, depending on the length of a minimal path, which will be at most $n^{*}$, and the direction of that path; call this the ``directed distance" between $a$ and $b$. If $n$ is odd, any pair of distinct nodes has one of $2n^{*}+1 = n$ types, depending on the length of a minimal path, which will be at most $n^{*}$, and the direction of that path if its length is $< n^{*}$; if the length of a minimal path is $n^{*}$, $a$ and $b$ will have an bi-directed distance of $n^{*}$, and $a$ and $b$ will otherwise have a directed distance of $< n^{*}$.

    The type of a set $S$ of distinct nodes will be determined by the distances between any two elements of $S$. We give necessary and sufficient criteria for an assignment of distances to pairs of nodes in $S$ to be consistent:

    (1) If a directed path of length $k \leq \frac{n}{2}$ is indicated by a chain of directed distances from $a$ and $b$ with total length $d$ (i.e. there are $a =a_{0}, a_{1}, \ldots, a_{m-1}, a_{m} = b $, with a directed distance of $d_{i}$ from $a_{i-1}$ to $a_{i}$ and $d=\sum_{i =1}^{n}d_{i} \leq \frac{n}{2}$), then the directed distance from $a$ to $b$ is at most $d$.

    (2) Chains of distances distances cannot indicate a directed cycle of length $\leq n$ (i.e. $a =a_{0}, a_{1}, \ldots, a_{m-1}, a_{m} = a $, with a directed (or bi-directed) distance of $a_{i}$ from $a_{i-1}$ to $a_{i}$ and $d=\sum_{i =1}^{n}d_{i} \leq n$).

    Clearly (1) and (2) are necessary. To show they are sufficient, assume distances are assigned to pairs in a set $S$; we will find a model $M$ containing $S$ realizing those assignments. For each pair $a, b \in S$, if a directed distance of $d$ (or an bi-directed distance of $d = n^{*}$) is assigned from $a$ to $b$, draw a path from $a$ to $b$ of length $d$ as well as a path in the opposite direction of length $n+1-d$. We show that $S$ together with these new vertices has no directed cycles of length $\leq n$. Suppose it contained such a cycle. That cycle can be partitioned into the paths added between elements of $S$. Suppose it contained the longer of the two paths, of length $n - d +1$, added from, say, $b$ to $a$. It cannot contain the shorter path, of length $d$, between $a$ and $b$, as then this cycle would be too long. But perhaps there another path that is even shorter than the directed distance $d$ from $a$ to $b$, formed out of paths of length less than $d \leq n^{*}$ between other vertices in $S$. This cannot happen, by (1). (This also handles the case where $a$ and $b$ have an bi-directed distance of $n^{*}$ when $n$ is odd.) Otherwise, all of the added paths between elements of $S$ that make up the cycle of length $\leq n$, are the paths of length $< n^{*}$ going in the direction of the directed distances. But this cannot happen, by (2). Call the union of $S$ with the additional paths $T$, which we have shown has no cycles of length $\leq n$, and find a model $M$ containing $T$. Then any two nodes in $S$ will have directed distance in $M$ at most the assigned distance $d$, as we added the shorter path going in that direction, but no less than the assigned distance, as we added the longer paths of length $n - d$ in the other direction, so a new path of length $< d$ in the direction of the directed distance would create a cycle of length $\leq n$.

    We next show, as an analogue of Lemma \ref{5-fa}, that sets can be ``freely amalgamated:" if $A \cap B = C$ with $A$ and $B$ given consistent assignments of distances agreeing on $C$, then there is an assignment of distances on $C$ agreeing with that on $A$ and $B$ and so that for $a \in A \backslash C$, $b \in B \backslash C$, the distance between $a$ and $b$ is the least total length of a chain of directed distances of total length $\leq n^{*}$ between $A$ and $B$ going through $C$ (i.e., without requiring a directed distance between a point of $A \backslash C$ and a point of $B \backslash C$ as one of the steps in the chain), going in the direction of that chain, and is otherwise of length $n^{*}$. We first show that the directed distances already within $A \cup B$ (i.e. not between a point of $A \backslash C$ and a point of $B \backslash C$) satisfy (1) and (2) (i.e., if $a \in A$ then $b \in A$, and if $a_{i} \in A$ then $a_{i+1} \in A$, and similarly for $B$.) For (1), we can assume without loss of generality that $a$ and $b$ are in $A$, and then the chain can be broken up into parts in $A$ and parts in $B$ each going between two nodes of $C$, but all of the parts in $B$ can be presumed to be in $C \subseteq A$, by (1) on $B$; having assumed this, we can then apply $1$ on $A$. For (2), a directed cycle of length $\leq n$ formed by chains of directed distances can similarly be broken up into parts in $A$ and parts in $B$, but only one of the parts, say in $A$ can have length $\geq n^{*}$, and all of the other parts, by (1) in $A$ and $B$, can be presumed in $C \subseteq A$, contradicting (2) in $A$. So the direction of the chain of shortest total length $\leq n^{*}$ between a point of $A \backslash C$ and a point of $B \backslash C$ going through $C$ is well-defined, by (2) for chains in $A \cup B$ going through $C$, so if such a chain exists we use it as the definition for the directed distance, and otherwise we choose a distance in an \textit{arbitrary} direction of size $n^{*}$.
    
    It remains to show (1) and (2) on the whole of $A \cup B$. For (1), if $a, b$ are not both in $A$ or both in $B$, then any directed distance of length $\leq n^{*}$ on a chain from $a$ to $b$ between a point of $A \backslash C$ and a point of $B \backslash C$ can be replaced with a chain going through $C$ of the same total length in the same direction, by construction. So any chain of directed distances between $a$ and $b$ can be replaced by a chain of directed distances of the same total length between $a$ and $b$ and in the same direction, going through $C$, so the directed distance between $a$ and $b$ will be as short or shorter in that direction, by definition. To complete (1) on the whole of $A \cup B$, if $a, b$ both belong to, say, $A$, then again by construction we can assume a chain of directed distances to be a chain going through $C$, and then use (1) for chains going through $C$. Finally, for (2), suppose first that the distances indicating a $\leq n$ cycle are either between two points of $A$ or $B$ or added between a point of $A \backslash C$ and a point of $B \backslash C$ because of a chain of the same length and the same direction going through $C$. Then those distances can be replaced with those chains, reducing (2) to the case of chains going through $C$, see above. Otherwise, one of the distances must be $n^{*}$ between a point $a \in A \backslash C$ and a point $b \in B \backslash C$, added because there is no chain going through $C$ of length $\leq n^{*}$. But the rest of the cycle must be a chain of distances of total length $\leq n^{*}$, each of which, being of length $< n^{*}$ can be replaced with chains going through $C$, and can thus be assumed a chain of length $< n^{*}$ going through $C$ from $a$ to $b$, a contradiction to the assumption on $a$ and $b$ that caused their  distance to have length $n^{*}$.

    Say that $C \subset A, B$, sitting in a fixed model, have distance $\geq k$ over $C$ for $k \leq n^{*}$ if $A \cap B = C$ and any $a \in A \backslash C$ and $b \in B \backslash C$ have distance in the same length and direction as the directed chain of distances of minimal total length through $C$ in $A \cup B$ if that length is $< k$, and otherwise have distance of length $\geq k$. We show an analogue of Lemma \ref{5-distance}: if $C \subseteq A, B, D$ with $A \cup D$ and $B \cup D$ freely amalgamated over $D$ as in the above discussion, and $A$ and $B$ both of distance $\geq k$ from $D$ over $C$, then (a) $A$ has distance $\geq \mathrm{min}(2k, n^{*})$ from $B$ over $C$ and (b) $A \cup B$ has distance $\geq k$ from $D$ over $C$. To show $A$, if a point $a \in A \backslash C$ and $b \in B \backslash C$ have distance of length $< \mathrm{min}(2k, n^{*})$ (say, going from $a$ to $b$), then there must be a chain of length $< \mathrm{min}(2k, n^{*})$ going through $D$. This chain can be broken alternately into parts entirely in $A$ and entirely in $B$. Using the fact that the chain has length $\leq n^{*}$, all of the parts except for the first and the last can be assumed entirely in $D$, and then replaced by (1) with a single distance between two points in $D$; that distance and the distance from the second point in $D$ to $B$ can then be replaced by a single distance, giving a chain of length $2$, going from $a$ to $d \in D$ and then to $B$. One of the two distances, say from $a$ to $d$, must have length $< k$, so there must be a chain of distances with the same total length and direction from $A$ to $D$, going through $C$ in $A \cup D$. So we can replace the distance from $a$ to $d$ with a distance from $a$ to $c$ followed by from $c$ to $d$, and can then replace the distance from $c$ to $d$ and from $d$ to $b$ with a single distance from $c$ to $b$, yielding a chain of equal or shorter length from $a$ to $b$ going through $C$, as desired. Meanwhile, (b) is obvious.

    Next, we show, if $A\ind_{C} B$ denotes a distance of $\geq n^{*}$ (so, $A$ and $B$ are freely amalgamated over $C$; note that this amalgamation is not unique), it implies forking-independence. We first show it implies dividing-independence: let $I=\{B_{i}\}_{i < \omega}$ with $B_{0} = B$ be a $C$-indiscernible sequence. Assign distances on $A \cup I$ to give $AB_{i}$ the same quantifier-free type as $AB$, and so the assignment agrees with the actual one on $I$. It suffices to show this satisfies (1) and (2). If $b \in B_{i} \backslash C$ and $a \in A \backslash C$, then if there is a chain of distances of length $\leq n^{*}$ in $A \cup I$ going through $C$ between $a$ and $b$, then by breaking this chain up into parts in $A$ and parts in $I$, and replacing the parts in $I$ with parts in $B_{i}$, there is a chain of distances of total length at least as short going in the same direction in $A \cup B_{i}$ through $C$. So the distance between $a$ and $b$ in our chosen assignment is the length and direction of the chain of least total length between $a$ and $b$ in $A \cup I$ going through $C$, using the actual assignments on $A$ and $I$. For the other pairs $a \in A \backslash C$ and $b \in I \backslash C$, a distance of length $ n^{*}$ is chosen. But in constructing the free amalgam, we showed that an arbitrary choice of direction for distances of length $n^{*}$  (though there will only be a choice when $n$ is even) is allowed for pairs with no path of length $\leq n^{*}$ going through the base. So our chosen assignment is one instance of the construction of the free amalgam of $A$ and $I$ over $C$, so in fact satisfies (1) and (2).

    It remains to show right-extension for $\ind$, which will hold if it is transitive, that is for $C \subseteq A$ and 
  $C \subseteq B \subseteq D$, $A \ind_{C} B$ and $A \ind_{B} D$ implies $A \ind_{C} D$. Suppose that $a \in A \backslash C$, $d \in D \backslash C$ and $a$ and $d$ have a distance $< n^{*}$, say, going from $a$ to $d$. Then there is, as above, a chain of distances going from, say $a$ to $b \in B$ to $d$, of the same total length. There is also a chain going from $a$ to $c \in C$ to $d $ of the same total length as the distance between $a$ and $d$, because that distance is also $< n^{*}$, so following that with the distance from $d \in D \subseteq B$ to $b \in B$, we get a chain going through $C$ of the same total length as the distance from $a$ to $b$, as desired.

  Now let $\ind^{0}= \ind^{a}$, $\ind^{i}$ denote a distance of $\mathrm{min}(2^{i}, n^{*})$ over the base. Exactly as in Example \ref{5-frcg}, we can then show that $\ind^{\eth^{k}} = \ind^{a} $ when $2^{k+1} \geq n$.

\end{example}

\begin{example}
(Model companion of undirected graphs without odd cycles of length $\leq n$, for $n$ odd).

In \cite{She95}, Shelah shows that the theory $T_{n}$ of undirected graphs without odd cycles of length $\leq n$, for $n$ odd, has a model companion and is strictly $\mathrm{NSOP}_{n+1}$. (This theory is further developed in \cite{CSS99}.) Again, this theory has quantifier elimination in the language with binary relations for paths of length $k$ (\cite{She95}). Instead of directed distances in the previous example, we now have the minimal length of a path between two vertices, an undirected distance which will be $\leq n^{*} = \lceil \frac{n}{2}\rceil$. A similar analysis to example \ref{5-cyclefree} will hold in this theory. Note, however, that because it is never true that $n = 2^{k}$, $T_{n}$ in this case cannot be used to witness that there are $\mathrm{SOP}_{2^{k}}$ theories where $\ind^{\eth^{k}}$ is symmetric.
\end{example}

\section{Bounds for symmetry and transitivity}

We now show that $\mathrm{SOP}_{2^{k+1}+1}$ is required for $\ind^{\eth^{k}}$ to be symmetric. (See Theorem 6.2 of \cite{GFA} for a related result on $\mathrm{NSOP}_{4}$.) From this and the previous section will follow the second clause of this theorem:

\begin{theorem}\label{5-symm}
    Assume $\ind^{\eth^{n}}$ is symmetric for $n \geq 1$. Then $T$ is $\mathrm{NSOP}_{2^{n+1}+1}$. Thus $2^{n+1}+1$ is the least $k$ so that every theory where $\ind^{\eth^{n}}$ is symmetric is $\mathrm{NSOP}_{k}$. 
\end{theorem}

We state the construction; fix a Skolemization of $T$. Suppose $T$ is $\mathrm{SOP}_{2^{n+1}+1}$; we show $\ind^{\eth^{n}}$ is asymmetric. Let $R(x, y)$ witness this; then there is an indiscernible sequence $\{c^{*}_{i}\}_{i \in 3\mathbb{Z}}$ so that $\models R(c_{i}, c_{j})$ for $i < j$, but there are no $(2^{n+1}+1)$-cycles. Let $M = \mathrm{dcl}_{\mathrm{Sk}}(\{c^{*}_{i}\}_{i \in \mathbb{Z}} \cup \{c^{*}_{2\mathbb{Z}+i}\}_{i \in \mathbb{Z}})$, and let $c_{i} = c^{*}_{\mathbb{Z}+i}$ for $i \in \mathbb{Z}$. For $k \geq 1$, let $R_{k}(x, y) =: \exists x_{0} \ldots x_{n-2} R(x, x_{0}) \wedge \bigwedge^{n-3}_{i=0}R(x_{i}, x_{i+1}) \wedge R(x_{n-2}, y)$ (so $R_{1}(x, y) =: R(x, y)$ and $R_{2}(x, y) =: \exists x_{0} R(x, x_{0}) \wedge R(x_{0}, y)$. We find instances of $k$-$\eth$-dividing:

\begin{lemma}\label{5-cycledividing}
    Let $0 \leq k \leq n$. Then

    $$R^{k}(y_{0}, \ldots, y_{2^{k}-1}, c_{0}, \ldots, c_{2^{k}})=: \bigwedge^{2^{k}-1}_{i=0} R_{2^{n-k}}(c_{i}, y_{i}) \wedge R_{2^{n-k}}(y_{i}, c_{i+1}) $$ $k$-$\eth$-divides (and therefore $k$-$\eth$-forks) over $M$.
\end{lemma}

\begin{proof}
    By induction on $k$. For $k=0$, we show that $R_{2^{n}}(c_{0}, y) \wedge R_{2^{n}}(y, c_{1})$ divides over $M$, specifically by $\{c_{2i}c_{2i+1}\}_{i < \omega}$. That is, $\{R_{2^{n}}(c_{2i}, y) \wedge R_{2^{n}}(y, c_{2i+1})\}_{i < \omega}$ is inconsistent; suppose it is consistent, realized by $c$. Then $\models R_{2^{n}}(c, c_{1}) \wedge R(c_{1}, c_{2}) \wedge R_{2^{n}}(c_{2}, c)$. So there is a $(2^{n+1} +1)$-cycle, a contradiction.

    Now suppose the statement holds for $0 \leq k \leq n-1$; we prove it is true for $k+1$. Let $\{c^{i}_{j}\}^{i<\omega}_{0 \leq j \leq 2^{k+1}}$ be a sequence with $c^{0}_{j}=c_{j}$ so that $\{R^{k+1}(y_{0}, \ldots, y_{2^{k+1}-1}, c^{i}_{0}, \ldots, c^{i}_{2^{k+1}})\}_{i < \omega}$ is realized by $c'_{0}, \ldots, c'_{2^{k+1}-1}$; by the definition of $(k+1)$-$\eth$-dividing, it suffices to show that $\{c^{i}_{j}\}^{i < \omega}_{0 \leq j \leq 2^{k+1}}$ is not an $\ind^{\eth^{k}}$-Morley sequence over $M$. For $0 \leq i \leq 2^{k}-1$, in particular 
    
    $$\models R_{2^{n-(k+1)}}(c_{2i}^{1}, c'_{2i}) \wedge R_{2^{n-(k+1)}}(c'_{2i}, c^{0}_{2i+1}) \wedge R_{2^{n-(k+1)}}(c^{0}_{2i+1}, c'_{2i+1}) \wedge  R_{2^{n-(k+1)}}(c'_{2i+1}, c^{1}_{2(i+1)})$$ It follows that $\models R_{2^{n-k}}(c^{1}_{2i}, c^{0}_{2i +1}) \wedge R_{2^{n-k}}(c^{0}_{2i +1}, c^{1}_{2(i+1)})  $ for $0 \leq i \leq 2^{k} -1$. Therefore,  
    
    $$R^{k}(y_{0}, \ldots, y_{2^{k}-1}, c_{0}^{1}, \ldots, c^{1}_{2i}, \ldots, c^{1}_{2^{k+1}}) \in \mathrm{tp}(c^{0}_{1} \ldots c^{0}_{2i+1} \ldots c^{0}_{2^{k+1}-1} /Mc_{0}^{1} \ldots c^{1}_{2i} \ldots c^{1}_{2^{k+1}})$$ Because $c^{1}_{0} \ldots c^{1}_{2i} \ldots c^{1}_{2^{k+1}} \equiv_{M} c^{0}_{0} \ldots c^{0}_{2i} \ldots c^{0}_{2^{k+1}}=c_{0} \ldots c_{2i} \ldots c_{2^{k+1}} \equiv_{M} c_{0} \ldots c_{2^{k}}$, and $R^{k}(y_{0}, \ldots, y_{2^{k}-1}, c_{0}, \ldots, c_{2^{k}})$ $k$-$\eth$-divides over $M$ by the induction hypothesis, $R^{k}(y_{0}, \ldots, y_{2^{k}-1}, c_{0}^{1}, \ldots, c^{1}_{2i}, \ldots, c^{1}_{2^{k+1}})$ $k$-$\eth$-divides over $M$, so 
    
    $$c^{0}_{1} \ldots c^{0}_{2i+1} \ldots c^{0}_{2^{k+1}-1}\nind_{M}^{\eth^{k}}c_{0}^{1}, \ldots c^{1}_{2i} \ldots c^{1}_{2^{k+1}}$$ and therefore, 
    
    $$c^{0}_{0} \ldots c^{0}_{2^{k+1}}\nind_{M}^{\eth^{k}}c_{0}^{1}\ldots c^{1}_{2^{k+1}}$$ So $\{c^{i}_{j}\}^{i < \omega}_{0 \leq j \leq 2^{k+1}}$ is not an $\ind^{\eth^{k}}$-Morley sequence over $M$.

\end{proof}

It follows from the case $k=n$ of Lemma \ref{5-cycledividing} and an automorphism that 

$$\{c_{2i-1}\}_{-2^{n-1} < i \leq 2^{n-1}}\nind^{\eth^{n}}_{M}\{c_{2i}\}_{-2^{n-1}\leq i \leq 2^{n-1}}$$ So we have obtained an instance of $n$-$\eth$ independence. When $n=1$, so $c_{-1}c_{1} \nind^{\eth^{1}}_{M}c_{-2}c_{0}c_{2}$, this is one direction of the asymmetry: $ c_{-2}c_{0}c_{2}\ind^{\eth^{1}}_{M}c_{-1}c_{1}$. To show this, we extend $\{c_{i}\}_{i \in \mathbb{Z}}$ to an $M$-indiscernible sequence $\{c_{i}\}_{i \in \mathbb{Q}}$. Then by construction, $\{c_{-(1+i)}c_{1+i}\}_{i \in [0, 1)  }$ (note $c_{-(1+0)}c_{1+0}=c_{-1}c_{1}$) is a finitely satisfiable Morley sequence over $M$, indiscernible over $Mc_{-2}c_{0}c_{2}$. So $c_{-2}c_{0} c_{2}\ind^{\eth^{1}}_{M}c_{-1}c_{1}$ follows from the following fact, which is immediate from Fact 6.1 of \cite{GFA} (this is just a standard application of left extension for finite satisfiability, as in the proof of Proposition \ref{5-fd}):

\begin{fact}\label{5-coheirs}
Let $\{b_{i}\}_{i < \omega}$ be a finitely satisfiable Morley sequence over $M$ with $b_{0} = b$ so that $\{\varphi(x, b_{i})\}_{i < \omega}$ is consistent. Then $\varphi(x, b)$ does not $1$-$\eth$-fork over $M$.

\end{fact}

This concludes the $\mathrm{NSOP}_{5}$ case. When $n \geq 2$, we may not be able to obtain the desired finitely satisfiable Morley sequence. However, unlike the case where $n =1$, we would not need anything stronger than an $\ind^{\eth^{n-1}}$-Morley sequence to show that a formula does not $n$-$\eth$-fork over $M$, as opposed to just $n$-$\eth$-dividing over $M$, because $n$-$\eth$-forking already coincides with $n$-$\eth$-dividing (Proposition \ref{5-fd}). Still, we do not show that $ \{c_{2i}\}_{-2^{n}\leq i \leq 2^{n}} \ind^{\eth^{n}}_{M} \{c_{2i-1}\}_{-2^{n} < i \leq 2^{n}}$ using an  explicit $\ind^{\eth^{n-1}}$-Morley sequence. Rather, let $m \leq 2^{n-1}$ be least such that $$\{c_{2i-1}\}_{-m < i \leq m}\nind^{\eth^{n}}_{M}\{c_{2i}\}_{-m\leq i \leq m}$$ Then $m > 0$. Let $\bar{a} = \{c_{2i}\}_{-m< i < m}$, $\bar{b}= \{c_{2i-1}\}_{-m < i \leq m}$, and $\bar{c}=c_{-2m} c_{2m}$. Then $\bar{b}\nind_{M}^{\eth^{n}}\bar{a}\bar{c}$, $\bar{a}\ind_{M}^{\eth^{n}}\bar{b}$ by minimality of $m$ and the fact that $\overline{a}\overline{b} \equiv_{M} \{c_{2i-1}\}_{-(m-1) < i \leq m-1}\{c_{2i}\}_{-(m-1) \leq i \leq m-1}$, and $\mathrm{tp}(ab/cM)$ is finitely satisfiable over $M$ by construction. To show asymmetry of $\ind_{M}^{\eth^{n}}$, it remains to show that $\bar{a}\bar{c}\ind_{M}^{\eth^{n}}\bar{b}$. We use the proof technique from Claim 6.2 of \cite{NSOP2}. Let $\varphi(\bar{x}, \bar{z}, \bar{b}) \in \mathrm{tp}(\bar{a}\bar{c}/M\bar{b})$; we show it does not $n$-$\eth$-fork over $M$, for which it suffices that it not $n$-$\eth$-divide over $M$. More explicitly, $\models\varphi(\bar{a}, \bar{c}, \bar{b})$, so $\models\varphi(\bar{a}, \bar{y}, \bar{b}) \in \mathrm{tp}(\bar{c}/M\bar{a}\bar{b})$. By finite satisfiability, there is some $\bar{m} \in M$ so that  $\models\varphi(\bar{a}, \bar{m}, \bar{b})$. So $\varphi(\bar{x}, \bar{m}, \bar{b}) \in \mathrm{tp}(\bar{a}/M\bar{b})$. Because $\bar{a}\ind_{M}^{\eth^{n}}\bar{b}$, there is then some $\ind^{\eth^{n-1}}$-Morley sequence $\{\bar{b}_{i}\}_{i < \omega}$ over $M$ with $\bar{b}_{0} =\bar{b}$ so that $\{\varphi(\bar{x}, m, \bar{b}_{i})\}_{i < \omega}$ is consistent. A fortiori, $\{\varphi(\bar{x}, \bar{z}, \bar{b}_{i})\}_{i < \omega}$ is consistent. So $\varphi(\bar{x}, \bar{z}, \bar{b})$ does not $n$-$\eth$-divide over $M$, and as $\varphi(\bar{x}, \bar{z}, \bar{b}) \in \mathrm{tp}(\bar{a}\bar{c}/M\bar{b})$ was arbitrary, $\bar{a}\bar{c}\ind_{M}^{\eth^{n}}\bar{b}$. This concludes the proof of Theorem \ref{5-symm}.

In all of the examples of the previous section, $\ind^{\eth^{n}}=\ind^{a}$ has the following properties:

Right transitivity: $a\ind_{M}^{\eth^{n}}M'$ and $a\ind_{M'}^{\eth^{n}}b$ implies $a\ind_{M}^{\eth^{n}}b$

Left transitivity:

$M'\ind_{M}^{\eth^{n}}b$ and $a\ind_{M'}^{\eth^{n}}b$ implies $a\ind_{M}^{\eth^{n}}b$ 

when $M \prec M'$ are models. We show:

\begin{theorem}\label{5-trans}
    Assume $\ind^{\eth^{n}}$ is left or right transitive for $n \geq 1$. Then $T$ is $\mathrm{NSOP}_{2^{n+1}+1}$. Thus $2^{n+1}+1$ is the least $k$ so that every theory where $\ind^{\eth^{n}}$ is right transitive is $\mathrm{NSOP}_{k}$, and similarly for left transitivity.
\end{theorem}

The proof is easier, as to show, say, right transitivity fails, we  produce $a, b, M_{0} \prec M_{1} \prec \ldots \prec M_{k}$ so that $a\ind_{M_{i}}^{\eth^{n}}M_{i+1}$ for $0 \leq i \leq k$,  $a\ind_{M_{k}}^{\eth^{n}}b$, but $a\nind_{M_{0}}^{\eth^{n}}b$; the dependency $a\nind_{M_{0}}^{\eth^{n}}b$ will be a slight modification of the above to produce models, but the instances of independence, $a\ind_{M_{i}}^{\eth^{n}}M_{i+1}$, will arise directly from the construction, unlike in the proof for symmetry. Again, assume $R(x, y)$ gives an instance of $\mathrm{SOP}_{2^{n+1} + 1}$. Choose a Skolemization of $T$ and a $T^{\mathrm{Sk}}$-indiscernible sequence $\{c^{*}_{i}\}_{\mathbb{Z} + \mathbb{Z} \times \mathbb{Z} + \mathbb{Z}}$ so that $\models R(c_{i}^{*}, c_{j}^{*})$ for $i < j < \mathbb{Z} + \mathbb{Z} \times \mathbb{Z} + \mathbb{Z}$. Let $M_{0} = \mathrm{dcl}_{\mathrm{Sk}}(\{c^{*}_{i}\}_{i \in \mathbb{Z}}\cup \{c^{*}_{\mathbb{Z} + \mathbb{Z}\times\mathbb{Z} + i}\}_{i \in \mathbb{Z}} )$. For $i \in \mathbb{Z}\times \mathbb{Z}$, let $c'_{i} = c'_{\mathbb{Z} + i}$. Now $\mathbb{Z} \times \mathbb{Z}$ as a set of ordered pairs, ordered lexicographically, and define $c_{i}=\{c_{(i, j)}\}_{j \in \mathbb{Z}}$. Again, it follows from Lemma \ref{5-cycledividing} that 

$$\{c_{2i-1}\}_{-2^{n-1} < i \leq 2^{n-1}}\nind^{\eth^{n}}_{M_{0}}\{c_{2i}\}_{-2^{n-1}\leq i \leq 2^{n-1}}$$ To make the notation easier, let $a_{i} = c_{2(-2^{n-1}+1 +i)-1}$ for $0 \leq i < 2^{n}$; in other words, $a_{i}$ is the $i$th term of $\{c_{2i-1}\}_{-2^{n-1} < i \leq 2^{n-1}}$. Let $b_{i}=c_{2(-2^{n-1}+i)}$ for $0 \leq i \leq 2^{n}$. Let $a = \{a_{i}\}_{0 \leq i < 2^{n}}$ and $b = \{b_{i}\}_{0 \leq i \leq 2^{n}}$. Then $a \nind_{M_{0}}^{\eth^{n}} b$. For the right transitivity case, it suffices to find $M_{0} \prec M_{1} \prec \ldots \prec M_{k}$ so that $a\ind_{M_{i}}^{\eth^{n}}M_{i+1}$ for $0 \leq i \leq k=2^{n} + 1$,  $a\ind_{M_{k}}^{\eth^{n}}b$, despite having shown $a\nind_{M_{0}}^{\eth^{n}}b$. For $0 < i \leq 2^{n}+1$, let $M_{i} = \mathrm{dcl}_{\mathrm{Sk}}(\{b_{0} \ldots b_{i-1} \})$. We show that for $0 < i < 2^{n}+1$, $a\ind_{M_{i}}^{\eth^{n}}M_{i+1}$, and $a\ind_{M_{k}}^{\eth^{n}}b$. This follows directly from unwinding definitions and applying to $\{c_{i}^{*}\}_{i \in \mathbb{Z} + \mathbb{Z}\times \mathbb{Z} + \mathbb{Z}}$ the following claim:

\begin{claim}
    Let $\{e_{i}\}_{i \in I}$ be an indiscernible sequence in $T^{\mathrm{Sk}}$, where $I$ is a linear order. Let $J \subset I$ be a set with no greatest element. Let $I_{1}, I_{2} \subseteq I$ be such that every element of $I_{2}$ is above every element of $J$, and no element of $I_{1}$ is between any two elements of $I_{2}$. Let $e_{I_{2}} = \mathrm{dcl}_{\mathrm{Sk}}(\{e_{s}\}_{s \in J \cup I_{2}})$, $e_{J} = \mathrm{dcl}_{\mathrm{Sk}}(\{e_{s}\}_{s \in J})$, and let $e_{I_{1}} = \{e_{s}\}_{s \in I}$. Then $e_{I_{1}} \ind_{e_{J}}^{\eth^{n}} e_{I_{2}}$ (in $T$).
\end{claim}

\begin{proof}
    We may assume that any element of $I_{2}$ between two elements of $I_{2}$ is itself in $I_{2}$. Let $I_{-}$ be the set of elements of $I$ below all of the elements of $I_{2}$, and $I_{+}$ the set of elements of $I$ above all of the elements of $I_{2}$. Extend $\{e_{i}\}_{i \in I}$ to $\{e_{i}\}_{i \in I_{-} + \omega \times I_{2} + I_{+}}$ (so that $e_{I_{2}} =\{e_{s}\}_{s \in J \cup \{0\}\times I_{2}}$) so that it is still indiscernible in $T^{\mathrm{Sk}}$. For $i < \omega$, let $e_{I_{2}}^{i}=\mathrm{dcl}_{\mathrm{Sk}}(\{e_{s}\}_{s \in J \cup \{i\}\times I_{2}})$. Then $\{e^{i}_{I_{2}}\}_{i < \omega}$ is a finitely satisfiable Morley sequence over $e_{J}$ with $e_{I_{2}}^{0} = e_{I_{2}}$, because every element of $I_{2}$ is above every element of $J$ and $J$ has no greatest element. Moreover, because no element of $I_{1}$ is in between any two elements of $I_{2}$, $\{e^{i}_{I_{2}}\}_{i < \omega}$ is indiscernible over $e_{J}e_{I_{1}}$. So $e_{I_{1}} \ind_{e_{J}}^{\eth^{n}} e_{I_{2}}$ by Fact \ref{5-coheirs}.
\end{proof}

This completes the case of right transitivity. For left transitivity, we must find $M_{0} = M^{0} \prec M^{1} \prec \ldots \prec M^{k}$ so that for $0 \leq i \leq 2^{n}$, $M^{i+1}\ind_{M^{i}}^{n^{\eth}}b$, and $a\ind_{M^{k}}^{\eth^{n}} b$, despite having shown $a\nind_{M_{0}}^{\eth^{n}}b$. For $0 < i \leq 2^{n}$, let $M_{i} = \mathrm{dcl}_{\mathrm{Sk}}(\{a_{0} \ldots a_{i-1} \})$. Then for $0 \leq i \leq 2^{n}$, $M^{i+1}\ind_{M^{i}}^{u}b$, and $a\ind_{M^{k}}^{u} b$. This completes the case of left transitivity and thus the proof of Theorem \ref{5-trans}.

We conclude by asking whether the converse holds, giving us a theory of independence for $\mathrm{NSOP}_{2^{n+1}+1}$ theories:

\begin{question}\label{5-openquestion}
Does $\mathrm{NSOP}_{2^{n+1}+1}$ imply symmetry of $\ind^{\eth^{n}}$? Does it imply transitivity of $\ind^{\eth^{n}}$?

\end{question}

\bibliographystyle{plain}
\bibliography{refs}

\end{document}